\documentclass{amsart}
\usepackage{latexsym,amsxtra,amscd,ifthen,amsmath,color, multicol,hyperref,mathdots,fancyvrb}
\usepackage{amsfonts}
\usepackage{verbatim}
\usepackage{amsmath}
\usepackage{amsthm}
\usepackage{amssymb,changes}
\usepackage{graphicx} \usepackage{caption} \usepackage{subcaption}
\usepackage[notcite,notref,final]{showkeys}
 \usepackage[all,cmtip]{xy}

\numberwithin{equation}{section}

\theoremstyle{plain}
\newtheorem{theorem}{Theorem}[section]
\newtheorem{lemma}[theorem]{Lemma}

\newtheorem{proposition}[theorem]{Proposition}

\newtheorem{corollary}[theorem]{Corollary}

\theoremstyle{definition}
\newtheorem{definition}[theorem]{Definition}

\newtheorem{remark}[theorem]{Remark}

\makeatletter              
\let\c@equation\c@theorem  
\makeatother

\newcommand{\C}{\mathbb{C}}

\newcommand{\id}{\text{id}}
\newcommand{\End}{\text{End}}

\begin{document}

\title[Explicit representations of Sklyanin algebras at a point of order 2]{Explicit Representations of 3-dimensional Sklyanin algebras associated to a point of order 2}
\author{Daniel J. Reich and Chelsea Walton}

\address{Reich: Department of Mathematics, North Carolina State University, Raleigh, North Carolina 27695, USA}
\email{djreich@ncsu.edu }

\address{Walton: Department of Mathematics, Temple University, Philadelphia, Pennsylvania 19122, USA}
\email{notlaw@temple.edu}

\bibliographystyle{abbrv}

\begin{abstract}
The representation theory of a 3-dimensional Sklyanin algebra $S$ depends on its (noncommutative projective algebro-) geometric data: an elliptic curve $E$ in $\mathbb{P}^2$, and an automorphism $\sigma$ of $E$ given by translation by a point. Indeed, by a result of Artin-Tate-van den Bergh, we have that $S$ is module-finite over its center if and only if $\sigma$ has finite order. In this case, all irreducible representations of $S$ are finite-dimensional and of at most dimension $|\sigma|$. 

In this work, we provide an algorithm in Maple to directly compute all irreducible representations of $S$ associated to $\sigma$ of order 2, up to equivalence.  Using this algorithm, we compute and list these representations. 
To illustrate how the algorithm developed in this paper can be applied to other algebras, we use it to recover well-known results about irreducible representations of the skew polynomial ring $\mathbb{C}_{-1}[x,y]$.
\end{abstract}



\maketitle


\section{Introduction}

We work over the ground field $\C$. The motivation of this work is to study, up to equivalence, irreducible finite-dimensional representations (irreps) of {\it Sklyanin algebras}  $S$ of global dimension 3 [Definition~\ref{def:Sklyanin}]. Past work on this problem include results on bounds on the dimension of irreps of $S$ \cite{Walton}, and on a geometric parametrization of (trace-preserving) irreps of $S$ \cite{DeLaetLeBruyn}. The focus of this paper is to determine, for a class of Sklyanin algebras,  all  {\it explicit} irreps up to equivalence. Namely, we compute {\it irreducible matrix solutions} to the defining equations of $S$, up to an action of a general linear group. A geometric parametrization of the set of irreps of $S$ is also presented, as this is the typical approach to understanding aspects of Sklyanin algebras.

\begin{remark} \label{rem:intro} We directly compute the irreps via a Maple algorithm.  A more conceptual technique, using noncommutative projective algebraic geometry (and Clifford theory for these particular Sklyanin algebras), can be used to solve this problem.  We nevertheless hold to the computational approach because it can be adapted (much more easily in some cases) to other algebras; for further discussion of the complexity of this approach, see Remarks \ref{rem:apply} and~\ref{rem:apply2}.  
\end{remark}

To begin, let us define the algebra under investigation.

\begin{definition} \cite{ATV1} \label{def:Sklyanin} 
The {\it 3-dimensional Sklyanin algebra} $S:=S(a,b,c)$ over $\mathbb{C}$ is generated by three non-commuting variables $x$, $y$, $z$ subject to the following relations:
\begin{equation} \label{eq:Srelations}
ayz+bzy+cx^2~ = ~azx+bxz+cy^2 ~= ~axy+byx+cz^2  ~= ~0.
\end{equation}
Here,  $[a:b:c] \in \mathbb{P}_\mathbb{C}^2$, with $abc \ne 0$ and $(3abc)^3 \ne (a^3+b^3+c^3)^3$.  
\end{definition}

This algebra is rather resistant to noncommutative Gr\"{o}bner basis methods, that is, it is difficult to write down a $\mathbb{C}$-vector space basis of $S$ (consisting of monomials in $x,y,z$);  see, for instance,  \cite[Exercise~1.7]{Rogalski}. (The reader may also be interested in recent work of Iyudu-Shrakin \cite{IyuduShkarin}.) In fact, it is common practice to consider the geometric data of $S$ in the context of  Noncommutative Projective Algebraic Geometry \cite{ATV1,Rogalski, StaffordvandenBergh} to analyze its ring-theoretic behavior. By \cite[Equations~1.6 and~1.7]{ATV1}, the geometric data of $S(a,b,c)$ consists of an elliptic curve $E:=E_{a,b,c} \subset \mathbb{P}_\mathbb{C}^2$ defined the equation, 
\begin{equation} \label{eq:E}
E_{a,b,c}: ~~(a^3+b^3+c^3)(uvw)-(abc)(u^3+v^3+w^3) = 0,
\end{equation} 
and an automorphism of this elliptic curve $\sigma:=\sigma_{a,b,c}$ given by 
\begin{equation} \label{def:sigma}
\sigma_{a,b,c}([u:v:w]) = [acv^2-b^2uw~:~bcu^2-a^2vw~:~abw^2-c^2uv].
\end{equation}
Here, the automorphism is given by translation of the point $[a:b:c] \in E_{a,b,c}$, where $[1:-1:0]$ is the origin of $E_{a,b,c}$. The \textit{order} of $\sigma$, denoted $|\sigma|$, is the smallest $n \in \mathbb{N}$ such that $\sigma^n = \id_{E}$.  If no such $n$ exists, then $|\sigma|=\infty$.  Consider the following terminology.

\begin{definition}
We say that a Sklyanin algebra $S(a,b,c)$ is {\it associated to a point \textnormal{(}\text{$[a:b:c] \in E_{a,b,c}$}\textnormal{)} of order~$n$} if the automorphism $\sigma_{a,b,c}$ has order $n$.
\end{definition}

\noindent The role of this geometric data for our work will be explained towards the end of this section.

Now let us recall some basic representation theory terminology. Take $n$ to be a positive integer.
   An \textit{$n$-dimensional representation} of $S:=S{(a,b,c)}$ is an algebra homomorphism $\psi : S \to \End(V)$ where $V$ is a  $\mathbb{C}$-vector space of dimension $n$.  Since $\End(V)$ is isomorphic to $Mat_n(\mathbb{C})$, there is a one-to-one correspondence between the $n$-dimensional representations of $S{(a,b,c)}$ and the $n\times n$ {\it matrix solutions} $(X,Y,Z)$ to the system of equations \eqref{eq:Srelations}.  Here, $X = \psi(x)$, $Y= \psi(y)$, and $Z= \psi(z)$.
   
Next, we discuss irreducibility. Given a representation $\psi : S \to \End(V)$, a subspace $W$ of $V$ is called $S$-{\it stable} if $\psi(s)(w) \in W$, for all $s \in S$, $w \in W$. Such a subspace $W$ yields a {\it sub-representation} of $S$, given as $\psi': S \to \End(W)$. We say that  $\psi$ is \textit{irreducible} if the only $S$-stable subspaces of $V$ are $\{0\}$ and itself, that is, if there are no proper sub-representations $\psi'$ of $\psi$. Similarly, there is a notion of irreducibility for a matrix solution $(X,Y,Z)$ to equations \eqref{eq:Srelations}; see Lemma~\ref{lem:irred}. 

Now we recall when two representations/ matrix solutions of $S$ are equivalent.
We say that $n$-dimensional representations $\psi, \phi: S \to \End(V)$  are {\it equivalent} if there exists a matrix $Q \in GL_n(\C)$ so that $\psi(s) = Q \phi(s) Q^{-1}$, for all $s \in S$. Likewise, two matrix solutions $(X_0,Y_0,Z_0)$ and $(X_1,Y_1,Z_1)$ to \eqref{eq:Srelations} are \textit{equivalent} if there exists $Q \in GL_n(\C)$ such that $Q^{-1}X_0Q=X_1$, $Q^{-1}Y_0Q=Y_1$, and $Q^{-1}Z_0Q=Z_1$. Note that two equivalent representations/ matrix solutions are  either both irreducible or both reducible.

As the reader can imagine, studying explicit finite-dimensional representations of the algebras $S(a,b,c)$ is difficult computationally. Now by \cite[Theorem~1.3]{Walton}, we only have non-trivial finite-dimensional representations of $S$ when the automorphism $\sigma$ of  \eqref{def:sigma}  has finite order.  So, we refine our goal: we study the irreps of $S(a,b,c)$ associated to a point $[a:b:c] \in E_{a,b,c}$ of order 2. (Note that the order 1 case is precisely the case when $S$ is commutative [Lemma~\ref{order1}].)

\begin{lemma}[Lemma~\ref{lem:order2}]
A Sklyanin algebra $S(a,b,c)$ is associated to a point $[a:b:c] \in E_{a,b,c}$ of order~2 if and only if $a=b$. \end{lemma}

In this case,  we assume that $a = b=1$ by rescaling. Therefore, our goal is to study the representation theory of the 3-dimensional Sklyanin algebra $S(1,1,c)$, where by Definition~\ref{def:Sklyanin}, $c \neq 0, ~c^3 \neq 1, -8$. By Lemma~\ref{lem:irrep2}, all 1-dimensional irreps of $S(1,1,c)$ are trivial, and all irreps of $S(1,1,c)$ are finite-dimensional, of at most dimension 2. Thus, we only need to compute the irreps of dimension 2; we achieve this as follows.

\begin{theorem}  \label{thm:intro}
The non-trivial explicit irreps (or matrix solutions) of the 3-dimensional Sklyanin algebra $S(1,1,c)$ are of dimension 2. They are classified up to equivalence; the representatives of equivalence classes of irreps of $S(1,1,c)$ are provided in Tables 3 and 4 in Sections~\ref{sec:1Jordanblock} and~\ref{sec:2Jordanblock}, respectively.
\end{theorem}

 In Section~\ref{sec:prelim}, we provide background material and some preliminary results. In Section~\ref{sec:methodology}, we give an outline (Steps 0-2, 3a, 3b) of our algorithm to prove Theorem~\ref{thm:intro}. 
The algorithm then begins in Section~\ref{sec:allreps}, where we determine all of the 2-dimensional representations of $S(1,1,c)$, and exclude `families' of reducible representations; this is Steps 0-2 of the algorithm. In Sections~\ref{sec:1Jordanblock} and~\ref{sec:2Jordanblock}, we determine representatives of equivalence classes of 2-dimensional irreps of $S(1,1,c)$; this is Steps~3a and~3b of the algorithm.

The study of the irreps of $S(1,1,c)$ ends in Section~\ref{sec:geometry}, where for completion, we discuss a geometric parametrization of equivalence classes of  irreps of $S(1,1,c)$ (e.g., we illustrate the {\it Azumaya locus} of $S(1,1,c)$ over the center of $S(1,1,c)$). Namely we have the result below.

\begin{theorem}[Theorem~\ref{thm:geometry}]
The set of equivalence classes of irreps of $S(1,1,c)$ is in bijective correspondence with the points of  the 3-dimensional affine variety:
$$X_c:= \mathbb{V}(g^2 - c^2(u_1^3 + u_2^3 + u_3^3) - (c^3-4)u_1 u_2 u_3) \subseteq \mathbb{C}^4_{\{u_1,u_2,u_3,g\}}.$$
In particular, $X_c \setminus \{\underline{0}\}$ is the Azumaya locus of $S$ over its center (i.e., points of $X_c \setminus \{\underline{0}\}$ correspond to  2-dimensional irreps of $S$), and the origin of $X_c$ corresponds to the trivial representation of $S$.
\end{theorem}

\begin{remark} \label{rem:apply}
We would like to point out that one can adjust our algorithm to prove Theorem~\ref{thm:intro} to examine equivalence classes of irreps of other algebras with generators and relations, especially those that are module-finite over their center. Although, the run-time and complexity of the output of the algorithm is in direct correlation with the number of generators and relations of the algebra, along with the algebra's {\it polynomial identity degree} (if applicable).
\end{remark}

We illustrate the remark above in  Section~\ref{sec:skewpoly}, where we tailor our algorithm to examine irreps of the following skew polynomial ring:
$$\mathbb{C}_{-1}[x,y] := \mathbb{C}\langle x,y \rangle/(xy+yx).$$ 
Like $S(1,1,c)$, it is well-known that all irreps of $\mathbb{C}_{-1}[x,y] $ are finite-dimensional,  of dimension at most~2 [Lemma~\ref{lem:skew}(c)]. See Proposition~\ref{prop:skew} and Corollary~\ref{cor:skew} for the results on the representation theory of $\mathbb{C}_{-1}[x,y]$.

\medskip

Unless stated otherwise, computational results in this work are performed with the computer algebra system Maple\texttrademark (version~16). All code will be presented in ~{\tt typewriter typeface}, and are available on the authors' websites. 
\footnote{Reich: \url{http://www4.ncsu.edu/~djreich/index.html}. \hfill Walton: \url{https://math.temple.edu/~notlaw/research.html}.}

\begin{remark} \label{rem:apply2}
Part of the novelty of this work is that we obtain noncommutative algebraic/ representation theoreric results with Maple, which is a computer algebra system that is used typically  for commutative computations. We hope that in the future the task of determining equivalence classes of irreps of noncommutative algebras (presented by generators and relations) can be achieved easily using a computer algebra system that handles noncommutative Gr\"{o}bner bases, such as GAP \cite{GAP-GBNP}.
\end{remark}


\section{Preliminaries} \label{sec:prelim}

We begin with a result on the irreducibility of a representation/ matrix solution of a Sklyanin algebra $S=S(a,b,c)$. This result is well-known, and we will use it often without mention.
  
 \begin{lemma} \label{lem:irred} Let $\psi : S \to \End(V)$ be an $n$-dimensional representation of $S$, with corresponding matrix solution $(X,Y,Z)$ to the system of equations \eqref{eq:Srelations}. Then, the following are equivalent:
 \begin{enumerate}
 \item $\psi$ is irreducible;
 \item the corresponding $S$-module $V$ (where $S$ acts on $V$ via $\psi$) is simple;
 \item $\psi$ is surjective;
 \item $\psi(S)$ generates $\End(V) \cong Mat_n(\C)$ as a $\C$-algebra; and
\item every matrix in $Mat_n(\C)$ can be expressed as a noncommutative polynomial in $(X,Y,Z)$ over $\C$. \qed
 \end{enumerate}
 \end{lemma}
 
\noindent If any of the above conditions hold, we say that the matrix solution $(X,Y,Z)$ is {\it irreducible}.
 
On the other hand, we can determine when a matrix solution of $S$ is reducible by using Lemma~\ref{lem:irred}.

\begin{corollary} \label{cor:irred} An $n \times n$ matrix solution $(X,Y,Z)$ to  \eqref{eq:Srelations} (corresponding to a representation $\psi$ of $S$) is reducible if and only if there exists a subspace $W$ of $V$ of dimension $m<n$ with $X \cdot w, Y \cdot w, Z \cdot w \in W$ for all $w \in W$. Here, we embed $W$ into $V$ so that $\cdot$ is given by matrix multiplication. \qed
\end{corollary}

If $S$ is a Sklyanin algebra associated to a point of infinite order, then by \cite[Theorem~1.3(i)]{Walton}, we have that all finite-dimensional irreps of $S$ are trivial. On the other hand, Sklyanin algebras associated to points of finite order have an interesting representation theory, due to the following result.

\begin{proposition} \label{prop:irrep} Let $S(a,b,c)$ be a Sklyanin algebra associated to a point of finite order. Then, all irreducible representations of $S(a,b,c)$ are finite-dimensional, of at most dimension $|\sigma_{a,b,c}|$.
\end{proposition}

\begin{proof}
In this case, we have that the Sklyanin algebra $S(a,b,c)$ is module-finite over its center by \cite[Theorem~7.1]{ATV2}. Further, $S(a,b,c)$ has PI degree~$|\sigma_{a,b,c}|$ by \cite[Proposition~1.6]{Walton}. Hence, the irreducible representations of $S(a,b,c)$ are all finite-dimensional by \cite[Theorem~13.10.3(a)]{MR}, of dimension at most~$|\sigma_{a,b,c}|$ by \cite[Proposition~3.1]{BG}.  
\end{proof}

Now we analyze  parameters $(a,b,c) \in \C^3 $ so that the automorphism $\sigma_{a,b,c}$ from \eqref{def:sigma} has finite order. Recall that two projective points $[m_1:m_2:m_3]$, $[n_1:n_2:n_3] \in \mathbb{P}^2_{\mathbb{C}}$ are equal if and only if $m_1n_2-m_2n_1 = m_1n_3-m_3n_1 = m_2n_3-m_3n_2 = 0$ if and only if $n_i = \lambda m_i$ for all $i = 1,2,3$, for some nonzero $\lambda \in \mathbb{C}$.
Omitting the conditions on parameters $a,b,c$ for now, it is worth noting the following the result.

\begin{lemma} \label{order1} The automorphism $\sigma_{a,b,c}$ from \eqref{def:sigma} has order 1 if and only if
$a=1, b=-1, c=0$. In this case,  
$S(1,-1,0)$ is the commutative polynomial ring $\C[x,y,z]$.
\end{lemma}

\begin{proof}
If $\sigma$ has order 1, then we obtain that $[acv^2-b^2uw~:~bcu^2-a^2vw~:~abw^2-c^2uv] = [u:v:w]$. Therefore, $bcu^2w - (a^2 +ab) vw^2 +c^2 uv^2 =0$, which (by taking the coefficient of $uv^2$) implies that $c=0$. Without loss of generality, take $a=1$. Now, $[-b^2uw:-vw:bw^2] = [u:v:w]$, and we must have that $b=-1$ since $-vw^2 = bvw^2$. Therefore, the forward direction holds. For the converse, note that
$\sigma_{1,-1,0}([u:v:w]) = [-uw:-vw:-w^2]  = [u:v:w]$, so $\sigma_{1,-1,0}$ has order 1. The last statement is clear.
\end{proof}

Consider the following  preliminary results about Sklyanin algebras associated to a point of order 2.

\begin{lemma} \label{lem:order2}
Take $S=S(a,b,c)$ to be a 3-dimensional Sklyanin algebra associated to the automorphism $\sigma_{a,b,c}$ of \eqref{def:sigma}.  Then, $|\sigma_{a,b,c}| = 2$ if and only if $a=b$. \end{lemma}

\begin{proof}
Without loss of generality, take $a=1$. Consider the following routine  with comments.

{\footnotesize \begin{multicols}{2}
\noindent Let $\sigma_{1,b,c}^\ell([u:v:w]) = [f_\ell:g_\ell:h_\ell]$, for $\ell =1,2$.
\begin{verbatim}
f1:=c*v^2-b^2*u*w:        g1:=b*c*u^2-v*w:
h1:=b*w^2-c^2*u*v:
f2:=c*g1^2-b^2*f1*h1:     g2:=b*c*f1^2-g1*h1:  
h2:=b*h1^2-c^2*f1*g1:
\end{verbatim}
\noindent We want $\sigma_{1,b,c}^2 = \id$, or equivalently, we need that $[f_2:g_2:h_2]$ =$ [u:v:w]$. Hence, we want the expressions $v_1, v_2, v_3$ below to be simultaneously zero for some $b$ and $c$.
\begin{verbatim}
v1:=u*g2-f2*v:    v2:=u*h2-f2*w:       v3:=v*h2-g2*w:
\end{verbatim}

\noindent By Definition~\ref{def:Sklyanin}, we  exclude $b=c=0$. Now we extract the coefficients of $v_1, v_2, v_3$ and solve for $b, c$. 
\begin{verbatim}
var:=[u,v,w];
Coeffs:=[coeffs(collect(v1,var,'distributed'),var), 
         coeffs(collect(v2,var,'distributed'),var), 
         coeffs(collect(v3,var,'distributed'),var)];
solve([op(Coeffs),b<>0,c<>0],{b,c});
>     {b = 1, c = c}
\end{verbatim}
\end{multicols}}

\noindent Hence, $b=1$ and there are no conditions on $c$ (other than those in Definition~\ref{def:Sklyanin}). 

The converse is clear by the computation above, but we can verify this directly. If $a=b=1$, then $\sigma_{1,1,c}([u:v:w]) = [cv^2-uw : cu^2-vw : w^2-c^2uv]$. So, 

{\footnotesize \[
\begin{array}{l}
\sigma^2_{1,1,c}([u:v:w]) \\
= [c(cu^2-vw)^2-(cv^2-uw)(w^2-c^2uv) : c(cv^2-uw)^2-(cu^2-vw)(w^2-c^2uv) : (w^2-c^2uv)^2-c^2(cv^2-uw)(cu^2-vw)]\\
= [u(c^3u^3+c^3v^3+w^3-3c^2uvw) :  v(c^3u^3+c^3v^3+w^3-3c^2uvw) : w(c^3u^3+c^3v^3+w^3-3c^2uvw)]\\
=[u:v:w],
\end{array}
\]}

\noindent as desired.
\end{proof}

Hence, to work with Sklyanin algebras $S(a,b,c)$ associated to a point of order 2, we take $a=b=1$.

\begin{lemma} \label{lem:irrep2} We have the following statements for the Sklyanin algebra $S(1,1,c)$.
\begin{enumerate}
\item The only $1$-dimensional representation of $S(1,1,c)$ is the trivial representation.
\item All irreducible representations of $S(1,1,c)$ are finite-dimensional,  of at most dimension equal to 2.
\end{enumerate}
\end{lemma}

\begin{proof}
(a) One can compute this directly, or by using the short routine below:

{\footnotesize
\begin{verbatim}
solve([x*y+y*x+c*z^2,y*z+z*y+c*x^2, z*x+x*z+c*y^2], [x,y,z]);
>     [[x = 0, y = 0, z = 0]]
\end{verbatim}
}

(b) This follows from Proposition~\ref{prop:irrep} and Lemma~\ref{lem:order2}. 
\end{proof}


\section{Methodology and Terminology} \label{sec:methodology}

In this section, we provide an outline of the algorithm used to prove Theorem~\ref{thm:intro}; see Sections~\ref{sec:allreps}-\ref{sec:2Jordanblock} for the full details. The goal is to obtain {\it irreducible representative families} of $S(1,1,c)$ as defined below.

\begin{definition}
 We say that a set of matrix solutions of the defining equations of $S(1,1,c)$ (or of equations~\eqref{eq:Srelations} with $a=b=1$) is a  {\it representative family of matrix solutions}, if no two members within the set are equivalent. Further, we call this set an {\it irreducible representative family } if all of its members are irreducible matrix solutions of $S(1,1,c)$.
\end{definition}

\begin{center}
Note that we aim to have the parameter $c$ of $S(1,1,c)$ free.\\ So due to Maple's default alpha ordering, {\bf we refer to $c$ as {\tt zc} in the code below}.
\end{center}
\medskip

First, we make the following simplification.   
\smallskip

\noindent {\bf Step 0: Assume that the matrix $X$ is in Jordan form.} Due to Lemma~\ref{lem:irrep2} we know that all non-trivial irreps of $S(1,1,c)$ are of dimension $2$.  Hence, we only study $2\times 2$ matrix solutions $(X,Y,Z)$ of \eqref{eq:Srelations} with $(a,b,c)=(1,1,c)$. Initially, the entries of $X, Y, Z$ are $x_\ell, y_\ell, z_\ell$ for $\ell = 1,2,3,4$. We further simplify the problem by assuming that $X$ is in Jordan form.  This simplification is made because we wish to classify the irreps up to equivalence, and equivalence is determined by simultaneous conjugation by an invertible matrix.  So, we take $X$ to be either a single $2\times 2$ Jordan block or diagonal so that we have 3 or 2 less unknowns, resp.  We consider these cases separately.  
\medskip

\noindent {\bf Step 1: Find all families of matrix solutions.} Now, we solve \eqref{eq:Srelations} with $(a,b,c)=(1,1,c)$ for $2\times 2$ matrices $(X,Y,Z)$. The output consists of 2-dimensional {\it (matrix solution) families} of $S(1,1,c)$.   The solutions are grouped according to the default behavior of Maple.  We refer to these groups as {\tt Families}.
\medskip

\noindent {\bf Step 2: Eliminate reducible matrix solutions.} We run this step now to cut down on the run-time of the algorithm and the complexity of its output. Given a family of matrix solutions, we use Corollary~\ref{cor:irred} to determine if all members of this family are reducible. Namely, we let {\tt w =<<p,q>>} be a basis of a 1-dimensional subspace $W$ of $\mathbb{C}^2$. Note that if $p=p_1 + p_2 i$ and $q = q_1 + q_2 i$, for $i := \sqrt{-1}$ and  $p_1, p_2, q_1, q_2 \in \mathbb{R}$, then  $(p,q) \neq (0,0)$ precisely when $p \bar{p} + q \bar{q} \neq 0$. We examine when $W$ is stable under the action of $S(1,1,c)$; namely, we need $Xw, Yw, Zw$ to be a scalar multiple of $w$. So, we solve for $p$, $q$ subject to the following conditions: 
\begin{itemize}
\item $W$ is not the zero subspace \dotfill {\tt p*conjugate(p)+q*conjugate(q)<>0}
\item $XW \subset W$ \dotfill {\tt p*Xw[2][1]-q*Xw[1][1]} = 0
\item $YW \subset W$ \dotfill {\tt p*Yw[2][1]-q*Yw[1][1]} = 0
\item $ZW \subset W$ \dotfill {\tt p*Zw[2][1]-q*Zw[1][1]} = 0
\item conditions on $c$.
\end{itemize}
If there is a solution, then this implies that all members of the specified family are reducible. We remove such families from further computations by forming a list {\tt NonRedFams} consisting of families for which there is no $p,q$ satisfying the conditions above.

\medskip

\begin{center} 
{\it Steps 3a and 3b are independent of each other, and either can be run after Step 2.}
\end{center}

\medskip

\noindent {\bf  Step 3a: Account for equivalence between families.} 
For the remaining families of matrix solutions, we determine conditions when members of one family {\tt NonRedFams[i]} is equivalent to members of another family {\tt NonRedFams[j]}. These conditions are collected in the list {\tt BetweenFams}. 

We do so as follows. First, we force variables of {\tt NonRedFams[i]} to be in terms of $u_\ell, v_\ell, w_\ell$ instead of $x_\ell, y_\ell, z_\ell$, for $\ell=1,2,3,4$; this is executed with {\tt eval(NonRedFams[...],ChangeVars)}. Next, we conjugate the relabeled matrices simultaneously by a $2 \times 2$ matrix {\tt Q} to form {\tt Xconj}, {\tt Yconj}, {\tt Zconj}. Then, we solve for variables $u_\ell, v_\ell, w_\ell, x_\ell, y_\ell, z_\ell$ subject to the following conditions:
\begin{itemize}
\item {\tt Xconj} is equal to the $X$-matrix {\tt Xj} of  {\tt NonRedFams[j]} \dotfill {\tt Equiv1} = 0
\item {\tt Yconj} is equal to the $Y$-matrix {\tt Yj} of  {\tt NonRedFams[j]}\dotfill {\tt Equiv2} = 0
\item {\tt Zconj} is equal to the $Z$-matrix {\tt Zj} of  {\tt NonRedFams[j]}\dotfill {\tt Equiv3} = 0
\item conditions on $c$ and invertibility of {\tt Q}.
\end{itemize}
The output is {\tt [i, j, \{}conditions on $u_\ell, v_\ell, w_\ell, x_\ell, y_\ell, z_\ell${\tt \} ]},  which we interpret as follows.
\medskip

\noindent {\bf Interpretation:}  We can eliminate {\tt NonRedFams[i]} from our consideration if all of its members are equivalent to members of {\tt NonRedFams[j]} for some $j \neq  i$. This occurs if we get an output 
\begin{center}
\begin{tabular}{ll}
{\tt [i, j,...\{ }{each of $u_\ell, v_\ell, w_\ell$ is free }{\tt \}...]} &for $i <j$, or\\
{\tt [j, i,...\{ }{each of $x_\ell, y_\ell, z_\ell$ is free }{\tt \}...]} &for $j < i$.
\end{tabular}
\end{center}

\noindent We obtain that {\tt NonRedFams[i]} forms a representative family if we get  output
\begin{center}
{\tt [i, i,...\{ }{restrictions on  $u_\ell, v_\ell, w_\ell, x_\ell, y_\ell, z_\ell$ }{\tt \}...]}
\end{center}
under one of the following conditions:

\smallskip

$\bullet$  (i) each of $x_\ell$, $y_\ell$, $z_\ell$ is free and \quad \quad  (ii) each of $u_\ell$, $v_\ell$, $w_\ell$ is free, or depends only on $x_\ell$, $y_\ell$, $z_\ell$; or 

$\bullet$ (i) each of $u_\ell$, $v_\ell$, $w_\ell$ is free and \hspace{.16in} (ii) each of $x_\ell$, $y_\ell$, $z_\ell$ is free, or depends only on $u_\ell$, $v_\ell$, $w_\ell$.

\smallskip

\noindent  In either case above, we  set the free variables in (ii) equal to 1 to obtain representative families.
Otherwise, a careful examination is needed. 

Conditions $u_\ell, v_\ell, w_\ell, x_\ell, y_\ell, z_\ell$ may depend on entries of the matrix {\tt Q}. In this case, we can conclude that such variables are free as long as this does not violate invertibility of {\tt Q}.

\medskip

\noindent {\bf Step 3b: Check for full irreducibility conditions.} Here, we run the same code as in Step 2 
except that we solve for $p,q$ along with all variables $x_\ell, y_\ell, z_\ell$. The conditions are collected in a list called {\tt IrConditions}.
If the output for {\tt NonRedFams[i]} is {\tt [i]} (or empty), then all members of {\tt NonRedFams[i]} are irreducible.


\section{Families of non-reducible representations of $S(1,1,c)$} \label{sec:allreps}

Here, we execute Steps 0-2 of the algorithm discussed in the previous section. Namely, we  find all 2-dimensional representations of $S(1,1,c)$  by determining $2\times 2$ matrix  solutions $(X,Y,Z)$ to \eqref{eq:Srelations} with $a=b=1$ . Here, $X$ is in Jordan form, either one Jordan block or two Jordan blocks (diagonal). Moreover, we eliminate the families of solutions for which all of its members are reducible.

\medskip
\noindent \framebox{{\bf Steps 0 and 1}} \quad We set up the defining equations. 
  \medskip
  
{\footnotesize
\begin{verbatim}
restart;                            with(LinearAlgebra):
\end{verbatim}
}
 \medskip

\noindent For 1 Jordan block, uncomment {\tt \#}. For 2 Jordan blocks, uncomment {\tt \#\#}.
 \medskip
 
{\footnotesize
\begin{verbatim}
#  X:= <<x1, 0|1, x1>>:  Y:= <<y1, y3|y2, y4>>:  Z:= <<z1, z3|z2, z4>>:
## X:= <<x1, 0|0, x4>>:  Y:= <<y1, y3|y2, y4>>:  Z:= <<z1, z3|z2, z4>>:
\end{verbatim}
}
 \medskip

\noindent Continue by entering the following. Again, we refer to $c$ by $zc$ in the code below. 
 \medskip
 
{\footnotesize
\begin{verbatim}
XX:= Multiply(X,X):                 XY:= Multiply(X,Y):                 XZ:= Multiply(X,Z):   
YX:= Multiply(Y,X):                 YY:= Multiply(Y,Y):                 YZ:= Multiply(Y,Z):
ZX:= Multiply(Z,X):                 ZY:= Multiply(Z,Y):                 ZZ:= Multiply(Z,Z):
Eq1:= convert(YZ+ZY+zc*XX,list):    Eq2:= convert(XZ+ZX+zc*YY,list):    Eq3:= convert(XY+YX+zc*ZZ,list):
\end{verbatim}
}
 \medskip

\noindent We enter conditions on $c$ and solve for $Eq1$, $Eq2$, $Eq3$, subject to these conditions, to get all 2-dimensional representations of $S(1,1,c)$. For 1 Jordan block, uncomment {\tt \#}. For 2 Jordan blocks, uncomment {\tt \#\#}.

 \medskip
 
{\footnotesize
\begin{verbatim}
Conditions:= [zc<>0,zc^3<>-8,zc^3<>1]:
# Vars:= {zc,x1,y1,y2,y3,y4,z1,z2,z3,z4}:
## Vars:= {zc,x1,x4,y1,y2,y3,y4,z1,z2,z3,z4}:
M:= solve([op(Eq1),op(Eq2),op(Eq3),op(Conditions)],Vars):
\end{verbatim}
}

 \medskip

\noindent We need to work with all values of roots in expressions for $M$, call it $L$.

 \medskip
 
 {\footnotesize
\begin{verbatim}
L:= []:
for i from 1 to nops([M]) do        T:=map(allvalues,{M[i]}):
for j from 1 to nops(T) do          L:=[op(L),T[j]]:
end do:       end do:
\end{verbatim}
}

\medskip

\noindent We build solution families from the list $L$.
\medskip

{\footnotesize
\begin{verbatim}
Families:=[]:
for i from 1 to nops(L) do          Families:=[op(Families),[eval(X,L[i]),eval(Y,L[i]),eval(Z,L[i])]]:
end do:
\end{verbatim}
}
\medskip

\vspace{.2in}
\noindent \framebox{{\bf Step 2}} \quad We now remove families whose members are all reducible.

\medskip

{\footnotesize
\begin{verbatim}
w:=<<p,q>>:                         NonRedFams:=[]:
for i from 1 to nops(Families) do
Xw:=Multiply(Families[i][1],w):     Yw:=Multiply(Families[i][2],w):     Zw:=Multiply(Families[i][3],w):
NonRed:=solve([p*conjugate(p)+q*conjugate(q)<>0, 
               p*Xw[2][1]-q*Xw[1][1], p*Yw[2][1]-q*Yw[1][1], p*Zw[2][1]-q*Zw[1][1], op(Conditions)],[p,q]):
	if NonRed=[] then                   NonRedFams:=[op(NonRedFams),[Families[i]]]:  
	end if:  end do:
\end{verbatim}
}
\medskip

\noindent The output of Steps 0-2 can be viewed by entering the following:
\medskip

{\footnotesize
\begin{verbatim}
for i from 1 to nops(NonRedFams) do    print(NonRedFams[i]):     end do:
\end{verbatim}
}

\medskip

\noindent Now, we obtain the  results below.
\smallskip

\begin{center}
{\footnotesize
\[
\framebox{\text{$
\begin{array}{llllll}
(1) 
& X=\begin{pmatrix}
0 & 1 \\
0 & 0 \end{pmatrix}
&&
Y= \begin{pmatrix}
-y_4 & \frac{y_4^2+(y_4^4-8y_4z_4^3)^{1/2}}{2cz_4^2} \\
-cz_4^2 & y_4 \end{pmatrix}
&&
Z = \begin{pmatrix}
-z_4 & 0 \\
\frac{c(y_4^2+(y_4^4-8y_4z_4^3)^{1/2})}{2}-cy_4^2 & z_4 \end{pmatrix}\\\\
(2) 
& X=\begin{pmatrix}
0 & 1 \\
0 & 0 \end{pmatrix} 
&&
Y= \begin{pmatrix}
-y_4 & -\frac{-y_4^2+(y_4^4-8y_4z_4^3)^{1/2}}{2cz_4^2} \\
-cz_4^2 & y_4 \end{pmatrix}
&&
Z = \begin{pmatrix}
-z_4 & 0 \\
-\frac{c(-y_4^2+(y_4^4-8y_4z_4^3)^{1/2})}{2}-cy_4^2 & z_4 \end{pmatrix}\\\\
(3) &
X = \begin{pmatrix}
0 & 1 \\
0 & 0 \end{pmatrix} 
&& 
Y= \begin{pmatrix}
\frac{\alpha}{cz_3} & -\frac{-cz_2^2z_3-cz_2z_4^2+\frac{2\alpha z_4}{cz_3}}{z_3} \\
-cz_2z_3-cz_4^2 & -\frac{\alpha}{cz_3} \end{pmatrix}
&& 
Z = \begin{pmatrix}
-z_4 & z_2 \\
z_3 & z_4 \end{pmatrix}\\\\
(4) &
X = \begin{pmatrix}
0 & 1 \\
0 & 0 \end{pmatrix} 
&& 
Y= \begin{pmatrix}
-\frac{\beta}{cz_3} & -\frac{-cz_2^2z_3-cz_2z_4^2+\frac{2\beta z_4}{cz_3}}{z_3} \\
-cz_2z_3-cz_4^2 & \frac{\beta}{cz_3} \end{pmatrix}
&& 
Z = \begin{pmatrix}
-z_4 & z_2 \\
z_3 & z_4 \end{pmatrix}\\\\
(5) &
X = \begin{pmatrix}
0 & 1\\
0 & 0 \end{pmatrix} 
&&
Y = \begin{pmatrix}
-y_4 & \frac{y_4^2}{cz_4^2}\\
-cz_4^2 & y_4 \end{pmatrix}
&&
Z = \begin{pmatrix}
-z_4 & \frac{2y_4}{cz_4}\\
0 & z_4 \end{pmatrix}.\\
\end{array}
$}}
\]
}
\end{center}
\begin{center}
{\small 
\[
\framebox{\text{$
\begin{array}{c}
\alpha = c^2z_2z_3z_4+c^2z_4^3+(3c^4z_2^2z_3^2z_4^2+3c^4z_2z_3z_4^4+c^4z_4^6-cz_3^3+c^4z_3^3z_2^3)^{1/2}\\
\beta = -c^2z_2z_3z_4-c^2z_4^3+(3c^4z_2^2z_3^2z_4^2+3c^4z_2z_3z_4^4+c^4z_4^6-cz_3^3+c^4z_3^3z_2^3)^{1/2}
\end{array}
$}}
\]
}
\end{center}
\smallskip
\begin{center}
{\small {\sc Table} 1. Output of Steps 0-2: {\tt NonRedFams} for $X$ is one Jordan block case}
\end{center}

\vspace{.1in}

\begin{center}
{\footnotesize
\[
\framebox{\text{$
\begin{array}{llllll}
(1) 
& X= \begin{pmatrix}
\frac{cz_4^2}{2y_4} & 0\\
0 & -\frac{cz_4^2}{2y_4} \end{pmatrix}
&&
Y= \begin{pmatrix}
-y_4 & -\frac{y_4^3-z_4^3}{y_4y_3}\\
y_3 & y_4 \end{pmatrix}
&&
Z = \begin{pmatrix}
-z_4 & -\frac{z_4(8y_4^3+c^3z_4^3)}{4y_4^2y_3}\\
0 & z_4 \end{pmatrix}\\\\

(2) 
& X= \begin{pmatrix}
-x_4 & 0\\
0 & x_4 \end{pmatrix}
&&
Y= \begin{pmatrix}
0 & 0\\
y_3 & 0 \end{pmatrix}
&&
Z = \begin{pmatrix}
0 & -\frac{cx_4^2}{y_3}\\
0 & 0 \end{pmatrix}\\\\

(3) 
& X= \begin{pmatrix}
\frac{cy_4^2}{2z_4} & 0\\
0 & -\frac{cy_4^2}{2z_4} \end{pmatrix}
&&
Y= \begin{pmatrix}
-y_4 & -\frac{y_4(8z_4^3+c^3y_4^3)}{4z_3z_4^2}\\
0 & y_4 \end{pmatrix}
&&
Z = \begin{pmatrix}
-z_4 & \frac{y_4^3-z_4^3}{z_4z_3}\\
z_3 & z_4 \end{pmatrix}\\\\

(4) 
& X= \begin{pmatrix}
-x_4 & 0\\
0 & x_4 \end{pmatrix}
&&
Y= \begin{pmatrix}
0 & -\frac{cx_4^2}{z_3}\\
0 & 0 \end{pmatrix}
&&
Z = \begin{pmatrix}
0 & 0\\
z_3 & 0 \end{pmatrix}\\\\

(5) 
& X= \begin{pmatrix}
\frac{\gamma}{c^2y_3z_3} & 0\\
0 & -\frac{\gamma}{c^2y_3z_3} \end{pmatrix}
&&
Y= \begin{pmatrix}
-y_4 & -\frac{\frac{2z_4\gamma}{c^2y_3z_3}+cy_4^2}{cy_3}\\
y_3 & y_4 \end{pmatrix}
&&
Z = \begin{pmatrix}
-z_4 & -\frac{-\frac{2\gamma y_4}{c^2y_3z_3}+cz_4^2}{cz_3}\\
z_3 & z_4 \end{pmatrix}\\\\

(6) 
& X= \begin{pmatrix}
-\frac{\delta}{c^2y_3z_3} & 0\\
0 & \frac{\delta}{c^2y_3z_3} \end{pmatrix}
&&
Y= \begin{pmatrix}
-y_4 & -\frac{\frac{2z_4\delta}{c^2y_3z_3}+cy_4^2}{cy_3}\\
y_3 & y_4 \end{pmatrix}
&&
Z = \begin{pmatrix}
-z_4 & -\frac{\frac{2\delta y_4}{c^2y_3z_3}+cz_4^2}{cz_3}\\
z_3 & z_4 \end{pmatrix}
\end{array}
$}}
\]
}
\end{center}

\begin{center}
{\small \[
\framebox{\text{$
\begin{array}{c}
\gamma = -z_3^2z_4-y_3^2y_4+(z_3^4z_4^2+2z_3^2z_4y_3^2y_4+y_3^4y_4^2+c^3y_3z_3^3y_4^2+c^3y_3^3z_3z_4^2-2c^3y_3^2z_3^2y_4z_4)^{1/2}\\
\delta = z_3^2z_4+y_3^2y_4+(z_3^4z_4^2+2z_3^2z_4y_3^2y_4+y_3^4y_4^2+c^3y_3z_3^3y_4^2+c^3y_3^3z_3z_4^2-2c^3y_3^2z_3^2y_4z_4)^{1/2}
\end{array}
$}}
\]
}
\end{center}
\smallskip

\begin{center}
{\small {\sc Table} 2. Output of Steps 0-2: {\tt NonRedFams} for $X$ is two Jordan block case}
\end{center}

\medskip


\section{Equivalence  and Irreducibility: one Jordan block case} \label{sec:1Jordanblock}

We wish to classify the matrix solutions from Steps 0-2 (in the previous section) up to equivalence and extract the irreducible equivalence classes.  So, we would like to know under what conditions is a matrix solution equivalent to a member of the same/different solution family.  We then specify conditions for which the representative of an equivalence class of  matrix solutions is irreducible. This achieved with Steps 3a and 3b, respectively, as described in Section~\ref{sec:methodology}.
In this section, we continue the algorithm of Section~\ref{sec:allreps} in the case when $X$ is one Jordan block.
\vspace{.2in}

\noindent \framebox{\bf Step 3a} \quad To execute Step 3a, as described in Section~\ref{sec:methodology}, enter the following:

\medskip

{\footnotesize
\begin{verbatim}
BetweenFams:=[]:                    ChangeVars:=[x1=u1,x4=u4,y1=v1,y2=v2,y3=v3,y4=v4,z1=w1,z2=w2,z3=w3,z4=w4]:
Q:=<<q1,q3|q2,q4>>:                 Qi:=MatrixInverse(Q):
for i from 1 to nops(NonRedFams) do
	Xconj:=Multiply(Q,Multiply(eval(NonRedFams[i][1][1],ChangeVars),Qi)):
	Yconj:=Multiply(Q,Multiply(eval(NonRedFams[i][1][2],ChangeVars),Qi)):
	Zconj:=Multiply(Q,Multiply(eval(NonRedFams[i][1][3],ChangeVars),Qi)):
	for j from i to nops(NonRedFams) do
		Xj:=NonRedFams[j][1][1]:            Yj:=NonRedFams[j][1][2]:             Zj:=NonRedFams[j][1][3]:
		Equiv1:= convert(Xj-Xconj,list):    Equiv2:= convert(Yj-Yconj,list):     Equiv3:= convert(Zj-Zconj,list):
		Conditions:= [zc<>0,zc^3<>-8,zc^3<>1,q1*q4-q2*q3<>0]:
		Equiv:= solve([op(Equiv1),op(Equiv2),op(Equiv3),op(Conditions)]):
		BetweenFams:=[op(BetweenFams),[i,j,Equiv]]:
	end do:     end do:
\end{verbatim}
}
\medskip

\noindent The output of Steps 0-3a can be viewed by entering the following:
\medskip

{\footnotesize
\begin{verbatim}
for i from 1 to nops(BetweenFams) do      print(BetweenFams[i]):         end do:
\end{verbatim}
}

\vspace{.2in}

\noindent \framebox{\bf Interpretation} \quad  Consider the snippets of output:

\medskip

{\footnotesize
\begin{Verbatim}[baselinestretch=0.7]
   [1, 2, {q1 = q1, q2 = q2, q3 = 0, q4 = q1,

                                       3     3        1/2
                                  -4 w4  + v4  - v4 %1                           4 q1 w4
           v4 = v4, w4 = w4, y4 = -----------------------, z4 = -w4,  zc = - ----------------}, 
                                         2     1/2                                 2     1/2
                                       v4  - %1                              q2 (v4  - %1   )
             4       3
     %1 := v4  - 8 w4  v4
                                                                            2
                                                                    zc q2 w4
   [1, 5, {q1 = q1, q2 = q2, q3 = 0, q4 = q1, v4 = 0, w4 = w4, y4 = ---------, z4 = w4, zc = zc}]
                                                                       q1
\end{Verbatim}
}

\noindent In the first snippet, one sees that with a choice of $q_1$ and $q_2$, the parameter $c$ can be considered free without violating the invertibility of $Q$.  We can also conclude that any member of {\tt NonRedFams[1]} is equivalent to a member of {\tt NonRedFams[2]}, except when $v_4^2-(v_4^4-8w_4^3v_4)^{1/2}=0$, or  equivalently when $v_4$ or  $w_4=0$.  From the second snippet of output, we see that any member of {\tt NonRedFams[1]} is equivalent to a member of {\tt NonRedFams[5]} when $v_4 = 0$.  Moreover by Table 1, we have that in {\tt NonRedFams[1]} $w_4$ (identified with $z_4$) cannot be $0$.
\medskip

$\bullet$ ~So, we exclude {\tt NonRedFams[1]} from further computation.
\medskip

\noindent Now consider another two snippets of output:
\medskip

{\footnotesize
\begin{Verbatim}[baselinestretch=0.7]
   [2, 4, {q1 = q1, q2 = q2, q3 = 0, q4 = q1, v4 = v4, w4 = w4, z2 = z2,

                          2                                                2
              (2 RootOf(_Z  + 1 + _Z) w4 q2 - q1 z2) q1           RootOf(_Z  + 1 + _Z) w4 q2 - q1 z2
         z3 = -----------------------------------------,    z4 = - ----------------------------------,
                                   2                                               q2
                                 q2
                              2
                2 (2 RootOf(_Z  + 1 + _Z) w4 q2 - q1 z2) q1
         zc = - -------------------------------------------}]
                         2      4       3    1/2    2
                      (v4  + (v4  - 8 w4  v4)   ) q2
                                                                            2
                                                                    zc q2 w4
   [2, 5, {q1 = q1, q2 = q2, q3 = 0, q4 = q1, v4 = 0, w4 = w4, y4 = ---------, z4 = w4,  zc = zc}]
                                                                       q1
\end{Verbatim}
}
\medskip

\noindent Through a choice of $q_1$ and $q_2$, we consider $c$ to be free in {\tt [2,4,...}.  We conclude that any member of {\tt NonRedFams[2]} is equivalent to a member of {\tt NonRedFams[4]} for all values of $v_4$ and $w_4$ except when $v_4^2+(v_4^4-8w_4^3v_4)^{1/2}~=~0$, or equivalently when $v_4$ or $w_4=0$.  From the second snippet of output, we see that if $v_4=0$, any member of {\tt NonRedFams[2]} is equivalent to a member of {\tt NonRedFams[5]}.  From Table 1, we see that $w_4$ (identified with $z_4$) in {\tt NonRedFams[2]} cannot be $0$.
\medskip

$\bullet$~ So, we exclude {\tt NonRedFams[2]} from further computation.
\medskip

\noindent Now take into account the following snippets of output:
\medskip

{\footnotesize
\begin{Verbatim}[baselinestretch=0.7]
    [3, 4, {q1 = q1, q2 = q2, q3 = 0, q4 = q1, w2 = w2, w3 = w3, w4 = w4,

                                     2        2
                   2 q1 w4 q2 + w2 q1  - w3 q2                 -w3 q2 + q1 w4
              z2 = ----------------------------, z3 = w3, z4 = --------------, zc = zc}]
                                 2                                   q1
                               q1

    [4, 5]
\end{Verbatim}
}
\medskip

\noindent This implies that {\tt NonRedFams[3]} is equivalent to {\tt NonRedFams[4]}.
\medskip

$\bullet$ ~So, we exclude {\tt NonRedFams[3]} from further computation.

$\bullet$ ~Further, no member of {\tt NonRedFams[4]} is equivalent to a member of {\tt NonRedFams[5]}.

\medskip

Finally, we determine when the remaining families are representative families. Consider:
\medskip

{\footnotesize
\begin{Verbatim}[baselinestretch=0.7]
   [4, 4, {q1 = q1, q2 = q2, q3 = 0, q4 = q1, w2 = w2, w3 = w3, w4 = w4,
   
                                     2        2
                   2 q1 w4 q2 + w2 q1  - w3 q2                 -w3 q2 + q1 w4
             z2 = ----------------------------, z3 = w3, z4 = --------------, zc = zc}]
                                 2                                   q1
                               q1

                           q1 (-y4 + v4)
   [5, 5, {q1 = q1, q2 = - -------------, q3 = 0, q4 = q1, v4 = v4, w4 = w4, y4 = y4, z4 = w4, zc = zc}]
                                   2
                              zc w4
\end{Verbatim}
}
\medskip

We get that a member of {\tt NonRedFams[5]} is  equivalent to another member of this family for any value of~$y_4$. Without loss of generality, set $y_4=1$.
\medskip

$\bullet$ ~So, {\tt NonRedFams[5]} is a representative family with $y_4=1$.

\medskip

\noindent In {\tt NonRedFams[4]}, we obtain any value for $z_4$, say $a$, by setting $q_2 = (w_4-a)q_1/w_3$. (Note that by Table~1, $z_3$, identified by $w_3$, is not equal to 0.) This choice of $q_2$ does not violate the invertibility of $Q$. Further, it is easy to check that in this case, $z_2=w_2$. Thus, without loss of generality, set $z_4 =1$

\medskip

$\bullet$ ~So, {\tt NonRedFams[4]} is a representative family with $z_4=1$.

\vspace{.2in}

\noindent \framebox{\bf Step 3b} \quad Given the results above, we only need to execute this step for {\tt NonRedFams[4]} and {\tt NonRedFams[5]}, but we complete this for the whole list {\tt NonRedFams} as follows:

\medskip

{\footnotesize
\begin{verbatim}
IrConditions:=[]:
for i from 1 to nops(NonRedFams) do
Xw:=Multiply(NonRedFams[i][1][1],w): Yw:=Multiply(NonRedFams[i][1][2],w): Zw:=Multiply(NonRedFams[i][1][3],w):
Ir:=solve([p*conjugate(p)+q*conjugate(q)<>0,
           p*Xw[2][1]-q*Xw[1][1],p*Yw[2][1]-q*Yw[1][1],p*Zw[2][1]-q*Zw[1][1], zc<>0,zc^3<>1,zc^3<>-8]):
IrConditions:=[op(IrConditions),[i,Ir]]:
end do:
\end{verbatim}
}

\medskip

\noindent To see the output, enter:
\medskip

{\footnotesize
\begin{verbatim}
for i from 1 to nops(IrConditions) do      print(IrConditions[i]):       end do:
\end{verbatim}
}

 \medskip
 
 \noindent One gets that, for each $i$, all members of {\tt NonRedFams[i]} are irreducible matrix solutions of $S(1,1,c)$.

 \vspace{.2in}

\noindent \framebox{\bf Conclusion} \quad By entering {\tt eval(NonRedFams[4],[z4=1]);} and {\tt eval(NonRedFams[5],[y4=1]);},
one obtains
the representatives of equivalence classes of irreducible matrix solutions $(X,Y,Z)$ of equations~\eqref{eq:Srelations}, where $X$ is assumed to be one Jordan block. The output is listed in the  table below.

\begin{center}
{\footnotesize
\[
\framebox{\text{$
\begin{array}{lllll}
X = \begin{pmatrix}
0 & 1 \\
0 & 0 \end{pmatrix} 
&& 
Y= \begin{pmatrix}
-\frac{\beta}{cz_3} & -\frac{-cz_2^2z_3-cz_2+\frac{2\beta}{cz_3}}{z_3} \\
-cz_2z_3-c & \frac{\beta}{cz_3} \end{pmatrix}
&& 
Z = \begin{pmatrix}
-1 & z_2 \\
z_3 & 1 \end{pmatrix}\\\\
X = \begin{pmatrix}
0 & 1\\
0 & 0 \end{pmatrix} 
&&
Y = \begin{pmatrix}
-1 & \frac{1}{cz_4^2}\\
-cz_4^2 & 1 \end{pmatrix}
&&
Z = \begin{pmatrix}
-z_4 & \frac{2}{cz_4}\\
0 & z_4 \end{pmatrix}\\
\end{array}
$}}
\]
}
\end{center}
\begin{center}
{\small 
\[
\framebox{\text{$
\begin{array}{c}
\beta = -c^2z_2z_3-c^2+(3c^4z_2^2z_3^2+3c^4z_2z_3+c^4-cz_3^3+c^4z_3^3z_2^3)^{1/2}
\end{array}
$}}
\]
}
\end{center}
\smallskip
\begin{center}
{\small {\sc Table} 3. Representatives of equivalences classes of 2-dimensional irreps of $S(1,1,c)$, when $X$ is one Jordan block}
\end{center}

\vspace{.1in}


\section{Equivalence and Irreducibility: two Jordan block case} \label{sec:2Jordanblock}

As in the one Jordan block case, we wish to classify the matrix solutions from Steps 0-2 (in Section~\ref{sec:allreps}) up to equivalence and extract the irreducible equivalence classes.  So, we would like to know under what conditions is a matrix solution equivalent to a member of the same/different solution family.  We then specify conditions for which the representative of an equivalence class of  matrix solutions is irreducible. This achieved with Steps 3a and 3b, respectively, as described in Section~\ref{sec:methodology}.
In this section, we continue the algorithm of Section~\ref{sec:allreps} in the case when $X$ is two Jordan blocks.
\vspace{.2in}

\noindent \framebox{\bf Step 3a} \quad To execute Step 3a, as described in Section~\ref{sec:methodology}, enter the code for Step 3a provided in Section~\ref{sec:1Jordanblock}. (The memory and time for this operation was 
27068.0 MB and 523.78 seconds, respectively.)
The output of Steps 0-3a can be viewed by entering the following:
\medskip

{\footnotesize
\begin{verbatim}
for i from 1 to nops(BetweenFams) do      print(BetweenFams[i]):         end do:
\end{verbatim}
}

\vspace{.2in}

\noindent \framebox{\bf Interpretation} \quad Consider the following snippet of output:

\medskip

{\footnotesize
\begin{Verbatim}[baselinestretch=0.7]                                                                   
                v3 q4
   [1, 1, {q1 = -----, q2 = 0, q3 = 0, q4 = q4, v3 = v3, v4 = y4, w4 = z4, y3 = y3, y4 = y4, z4 = z4, zc = zc} 
                 y3
\end{Verbatim}
}
\medskip

\noindent  Note that $y_3 \neq 0$  in {\tt NonRedFams[1]} by Table 2.
\medskip

$\bullet$~ So, {\tt NonRedFams[1]} is a representative family with $y_3$ (identified with $v_3$) is 1 without loss of generality.

\medskip

\noindent Now take:

\medskip

{\footnotesize
\begin{Verbatim}[baselinestretch=0.7]                                                                   
                 v3 q4
   [2, 2,  {q1 = -----, q2 = 0, q3 = 0, q4 = q4, u4 = x4, v3 = v3, x4 = x4, y3 = y3, zc = zc}]
                  y3
\end{Verbatim}
}
\medskip

\noindent  Note that $y_3 \neq 0$  in {\tt NonRedFams[2]} by Table 2.
\medskip

$\bullet$~ So, {\tt NonRedFams[2]} is a representative family with $y_3$ (identified with $v_3$) is 1 without loss of generality.

\medskip

\noindent Consider the output:

\medskip

{\footnotesize
\begin{Verbatim}[baselinestretch=0.7]                                                                                                           
                 w3 q4
   [3, 3,  {q1 = -----, q2 = 0, q3 = 0, q4 = q4, v4 = y4, w3 = w3, w4 = z4, y4 = y4, z3 = z3, z4 = z4, zc = zc}
                  z3
\end{Verbatim}
}
\medskip

\noindent  Note that $z_3 \neq 0$  in {\tt NonRedFams[3]} by Table 2.
\medskip

$\bullet$~ So, {\tt NonRedFams[3]} is a representative family with $z_3$ (identified with $w_3$) is 1 without loss of generality.

\medskip

\noindent Next, consider the snippet of output below:
\medskip

{\footnotesize
\begin{Verbatim}[baselinestretch=0.7]
                               2
                          zc x4  q3
   [2, 4, {q1 = 0, q2 = - ---------, q3 = q3, q4 = 0, u4 = -x4, v3 = v3, x4 = x4, z3 = z3, zc = zc}]
                            z3 v3
\end{Verbatim}
}
\medskip

\noindent By Table 2, we have that  $z_3 \neq 0$ for {\tt NonRedFams[4]}. So by the output above,  we get that any member of {\tt NonRedFams[4]} is equivalent to a member {\tt NonRedFams[2]}.
\medskip

$\bullet$~ We exclude {\tt NonRedFams[4]} from further computation.
\medskip

\noindent Consider the output:

\medskip

{\footnotesize
\begin{Verbatim}[baselinestretch=0.7]                                                                                                           
                v3 q4                                                  z3 v3
   [5, 5, {q1 = -----, q2 = 0, q3 = 0, q4 = q4, v3 = v3, v4 = y4, w3 = -----, w4 = z4,
                 y3                                                     y3                                                               

           y3 = y3, y4 = y4, z3 = z3, z4 = z4, zc = zc}
\end{Verbatim}
}
\medskip

\noindent  We have that $y_3 \neq 0$  in {\tt NonRedFams[5]} by Table 2. Without loss of generality, we can take $y_3$ (identified with $v_3$) to be 1. In this case, $w_3=z_3$.

\medskip

$\bullet$~ So, {\tt NonRedFams[5]} is a representative family with $y_3 =1$.

\medskip

\noindent Now let us take:
\medskip

{\footnotesize
\begin{Verbatim}[baselinestretch=0.7]
                v3 q4                                                  z3 v3
   [5, 6, {q1 = -----, q2 = 0, q3 = 0, q4 = q4, v3 = v3, v4 = y4, w3 = -----, w4 = z4, 
                 y3                                                     y3

           y3 = y3, y4 = y4, z3 = z3, z4 = z4, zc = zc}]
\end{Verbatim}
}
\medskip

\noindent Note that by Table 2, we have  $y_3 \neq 0$ for {\tt NonRedFams[6]}. So by the output above,  we get that any member of {\tt NonRedFams[6]} is equivalent to a member {\tt NonRedFams[5]}.
\medskip

$\bullet$~ We exclude {\tt NonRedFams[6]} from further computation.
\medskip

But we still need to analyze the equivalence between members of {\tt NonRedFams[1]}, {\tt NonRedFams[2]}, {\tt NonRedFams[3]}, {\tt NonRedFams[5]}. In this case, the output is easier to interpret if we run Step 3b before Step 3a again.
\vspace{.2in}

\noindent \framebox{\bf Step 3b} \quad Given the results above, we only need to execute this step for {\tt NonRedFams[1]}, {\tt NonRedFams[2]}, {\tt NonRedFams[3]}, {\tt NonRedFams[5]}, but we complete this for the whole list {\tt NonRedFams} by entering the code for Step 3b provided in Section~\ref{sec:1Jordanblock}. Consider the snippets:

\medskip

{\footnotesize
\begin{Verbatim}[baselinestretch=0.7]
              y4 q
   [1, {p = - ----, q = q, y3 = y3, y4 = y4, z4 = 0, zc = zc}]
               y3

   [2, {p = 0, q = q, x4 = 0, y3 = y3, zc = zc}]

              z4 q
   [3, {p = - ----, q = q, y4 = 0, z3 = z3, z4 = z4, zc = zc}]
               z3

                                            2
   [5, {p = 0, q = q, y3 = 0, y4 = RootOf(_Z  + 1 + _Z) z4, z3 = z3, z4 = z4, zc = zc},

              z4 q                       z4 y3
       {p = - ----, q = q, y3 = y3, y4 = -----, z3 = z3, z4 = z4, zc = zc}]
               z3                         z3
\end{Verbatim}
}

\medskip

\noindent We obtain that: 

$\bullet$~ members of {\tt NonRedFams[1]}, {\tt NonRedFams[2]}, {\tt NonRedFams[3]} are irreducible
 precisely when  $z_4 \neq 0$, \linebreak 
\indent \indent$x_4 \neq 0$, $y_4 \neq 0$, respectively, and
 
 $\bullet$~  members of {\tt NonRedFams[5]} are irreducible
 precisely when \{$y_3 \neq 0$, $y_4 \neq e^{\pm 2 \pi i/3} z_4$\} or \{$y_3 z_4 \neq y_4 z_3$\}.

\vspace{.2in}

\noindent \framebox{\bf Step 3a again} \quad We execute Step 3a again for 

\medskip

{\footnotesize
\begin{Verbatim}
NewNonRedFams[1]:=eval(NonRedFams[1],[y3=1]):   NewNonRedFams[2]:=eval(NonRedFams[2],[y3=1]):
NewNonRedFams[3]:=eval(NonRedFams[3],[z3=1]):   NewNonRedFams[4]:=eval(NonRedFams[5],[y3=1]):
\end{Verbatim}
}

\medskip

\noindent The code and output is:

\medskip

{\footnotesize
\begin{Verbatim}
NewBetweenFams:=[]:                 ChangeVars:=[x1=u1,x4=u4,y1=v1,y2=v2,y3=v3,y4=v4,z1=w1,z2=w2,z3=w3,z4=w4]:
Q:=<<q1,q3|q2,q4>>:                 Qi:=MatrixInverse(Q):
for i from 1 to 4 do
  Xconj:=Multiply(Q,Multiply(eval(NewNonRedFams[i][1][1],ChangeVars),Qi)):
  Yconj:=Multiply(Q,Multiply(eval(NewNonRedFams[i][1][2],ChangeVars),Qi)):
  Zconj:=Multiply(Q,Multiply(eval(NewNonRedFams[i][1][3],ChangeVars),Qi)):
for j from i+1 to 4 do
  Xj:=NewNonRedFams[j][1][1]:         Yj:=NewNonRedFams[j][1][2]:          Zj:=NewNonRedFams[j][1][3]:
  Equiv1:= convert(Xj-Xconj,list):    Equiv2:= convert(Yj-Yconj,list):     Equiv3:= convert(Zj-Zconj,list):
  Conditions:= [zc<>0,zc^3<>-8,zc^3<>1,q1*q4-q2*q3<>0]:
  Equiv:= solve([op(Equiv1),op(Equiv2),op(Equiv3),op(Conditions)]):
  NewBetweenFams:=[op(NewBetweenFams),[i,j,Equiv]]:
end do:     end do:

for i from 1 to nops(NewBetweenFams) do          print(NewBetweenFams[i]):         end do:
\end{Verbatim}
}

{\footnotesize
\begin{Verbatim}[baselinestretch=0.7]
 > [1, 2, {q1 = q1, q2 = q1 v4, q3 = 0, q4 = q1, v4 = v4, w4 = 0, x4 = 0, zc = zc},

                                2
          {q1 = -q3 v4, q2 = -v4  q3, q3 = q3, q4 = 0, v4 = v4, w4 = 0, x4 = 0, zc = zc}, 

                                    q1 (q1 - q4)                 q2
          {q1 = q1, q2 = q2, q3 = - ------------, q4 = q4, v4 = ----, w4 = 0, x4 = 0, zc = zc}]
                                         q2                      q1

   [1, 3, {q1 = 0, q2 = q2, q3 = q3, q4 = 0, v4 = -y4, w4 = -y4, y4 = y4, z4 = y4,

                             3   2          2
            zc = RootOf(q3 _Z  y4  + 8 q3 y4  + 4 q2)}, 
        
                                                                      2
          {q1 = 0, q2 = q2, q3 = q3, q4 = 0, v4 = -y4, w4 = -RootOf(_Z  + 1 + _Z) y4, y4 = y4,

                       2                                3   2          2                 2
         z4 = RootOf(_Z  + 1 + _Z) y4, zc = RootOf(q3 _Z  y4  + 8 q3 y4  - 4 q2 RootOf(_Z  + 1 + _Z) - 4 q2)}]

   [1, 4]       [2, 3]       [2, 4]       [3, 4]
\end{Verbatim}
}

\medskip

We obtain that  $z_4=0$ in {\tt NewNonRedFams[1]} precisely when any member of {\tt NewNonRedFams[1]} is equivalent to a member of {\tt NewNonRedFams[2]}. On the other hand, we have that  $x_4=0$ in {\tt NewNonRedFams[2]} precisely when any member of {\tt NewNonRedFams[2]} is equivalent to a member of {\tt NewNonRedFams[1]}. But members of {\tt NewNonRedFams[1]} and {\tt NewNonRedFams[2]} are reducible when $z_4 = 0$ and $x_4=0$, respectively.

Now by a choice of $q_2$, $q_3$, we can consider $c$ to be free in {\tt [1, 3, ...]}. So, we get that  $z_4 = \zeta y_4$ for $\zeta^3 = 1$ in {\tt NewNonRedFams[3]} precisely when any member of {\tt NewNonRedFams[3]} is equivalent to a member of {\tt NewNonRedFams[1]}.
\medskip

Putting this together we conclude that:
\medskip

$\bullet$~ {\tt NewNonRedFams[1]}={\tt eval(NonRedFams[1],[y3=1])} is an irreducible representative family when   $z_4 \neq 0$.
 
 $\bullet$~ {\tt NewNonRedFams[2]}={\tt eval(NonRedFams[2],[y3=1])} is an irreducible representative family when  $x_4 \neq 0$.
 
  $\bullet$~ {\tt NewNonRedFams[3]}={\tt eval(NonRedFams[3],[z3=1])} is an irreducible representative family when $y_4 \neq 0$,  \linebreak
 \indent \indent and  there is no overlap with {\tt NewNonRedFams[1]} when $z_4 \neq \zeta y_4$ for $\zeta^3 =1$.

$\bullet$\hspace{.03in}{\tt NewNonRedFams[4]}={\tt eval(NonRedFams[5],[y3=1])} is an irreducible representative family when \linebreak
\indent \indent $y_4 \neq e^{\pm 2 \pi i/3} z_4$ \and $z_4 \neq y_4 z_3$.

 \vspace{.2in}

\noindent \framebox{\bf Conclusion} \quad We obtain
the following representatives of equivalence classes of irreducible matrix solutions $(X,Y,Z)$ of equations~\eqref{eq:Srelations}, where $X$ is assumed to be two Jordan blocks.

\begin{center}
{\footnotesize
\[
\framebox{\text{$
\begin{array}{llllll}
X= \begin{pmatrix}
\frac{cz_4^2}{2y_4} & 0\\
0 & -\frac{cz_4^2}{2y_4} \end{pmatrix}
&&
Y= \begin{pmatrix}
-y_4 & -\frac{y_4^3-z_4^3}{y_4}\\
1 & y_4 \end{pmatrix}
&&
Z = \begin{pmatrix}
-z_4 & -\frac{z_4(8 y_4^3 + c^3 z_4^3)}{4 y_4^2}\\
0& z_4 \end{pmatrix}
& \text{ for  $z_4 \neq 0$}\\\\

X= \begin{pmatrix}
-x_4 & 0\\
0 & x_4 \end{pmatrix}
&&
Y= \begin{pmatrix}
0 & 0\\
1 & 0 \end{pmatrix}
&&
Z = \begin{pmatrix}
0& -c^2 x_4^2\\
0&0\end{pmatrix}
& \text{ for  $x_4 \neq 0$}\\\\

X= \begin{pmatrix}
\frac{cy_4^2}{2z_4} & 0\\
0 & -\frac{cy_4^2}{2z_4} \end{pmatrix}
&&
Y= \begin{pmatrix}
-y_4 & -\frac{y_4(8 z_4^3 + c^3y_4^3)}{4 z_4^2}\\
0 & y_4 \end{pmatrix}
&&
Z = \begin{pmatrix}
-z_4 & \frac{y_4^3 - z_4^3}{z_4}\\
1& z_4 \end{pmatrix}
& \text{ for  \begin{tabular}{l}
$y_4 \neq 0$\\ $z_4 \neq \zeta y_4$, $\zeta^3 =1$
\end{tabular}}\\\\

X= \begin{pmatrix}
\frac{\gamma}{c^2z_3} & 0\\
0 & -\frac{\gamma}{c^2z_3} \end{pmatrix}
&&
Y= \begin{pmatrix}
-y_4 & -\frac{\frac{2z_4\gamma}{c^2z_3}+cy_4^2}{c}\\
1 & y_4 \end{pmatrix}
&&
Z = \begin{pmatrix}
-z_4 & -\frac{-\frac{2\gamma y_4}{c^2z_3}+cz_4^2}{cz_3}\\
z_3 & z_4 \end{pmatrix}
& \text{ for  \begin{tabular}{l}
$y_4 \neq e^{\pm 2 \pi i/3} z_4$\\ $z_4 \neq y_4 z_3$
\end{tabular}}
\end{array}
$}}
\]
}
\end{center}
\begin{center}
{\small 
\[
\framebox{\text{$
\begin{array}{c}
\gamma = -z_3^2z_4-y_4+(z_3^4z_4^2+2z_3^2z_4y_4+y_4^2+c^3z_3^3y_4^2+c^3z_3z_4^2-2c^3z_3^2y_4z_4)^{1/2}
\end{array}
$}}
\]
}
\end{center}
\smallskip
\begin{center}
{\small {\sc Table} 4. Representatives of equivalences classes of 2-dimensional irreps of $S(1,1,c)$, when $X$ is two Jordan blocks}
\end{center}

\vspace{.1in}


\section{Geometric parametrization of irreducible representations of $S(1,1,c)$} \label{sec:geometry}

Since the Sklyanin algebra $S = S(1,1,c)$ is module finite over its center, we can use the center $Z$ of $S$ to  provide a geometric parametrization of the set of equivalence classes of irreducible representations of $S$. (Recall by Definition~\ref{def:Sklyanin}, $c \neq 0$, $c^3 \neq 1,-8$.) Namely, we depict the {\it Azumaya locus} of $S(1,1,c)$ over its center \cite[III.1.7]{BGbook}. We refer the reader to \cite{Smith} for an introduction to affine varieties.

\begin{theorem} \label{thm:geometry} Let $Z$ be the center of the Sklyanin algebra $S=S(1,1,c)$.
\begin{enumerate}
\item We have that $Z$ is generated by $u_1 = x^2$, $u_2 = y^2$, $u_3 = z^2$, $g = cy^3 + yxz - xyz - cx^3$, subject to the degree 6 relation: 
$$F:= g^2 - c^2(u_1^3 + u_2^3 + u_3^3) - (c^3-4)u_1 u_2 u_3 = 0.$$
\item The set of equivalence classes of irreducible representations of $S$ are in bijective correspondence with the set of maximal ideals of the center $Z$ of $S$. Here, a representative $\psi$ of an equivalence class of irrep of $S$ corresponds to $(\ker \psi) \cap Z$, a maximal ideal of $Z$.

\smallskip

\item The geometric parametrization of the set of equivalence classes of irreducible representations of $S$ is the 3-dimensional affine variety ({\it 3-fold})
$$X_c:=\mathbb{V}(F) \in \mathbb{C}^4_{\{u_1, u_2, u_3, g\}}.$$ 
In particular, $X_c \setminus \{\underline{0}\}$ is the Azumaya locus of $S$ over $Z$. Indeed, points of $X_c \setminus \{\underline{0}\}$ (the smooth locus of $X_c$) correspond to irreducible 2-dimensional representations of $S$, and the origin of $X_c$ corresponds to the trivial representation of $S$.
\end{enumerate}
\end{theorem}
\medskip

Taking a value of $c$, say $5$, we can visualize the 3-fold $X_c$ by taking  2-dimensional slices at various values of $u_1$. See Figure~\ref{fig:Xc=5} below.  The following images were generated with WolframAlpha \cite{Wolframalpha}.

\begin{figure}[h] 
\centering
\begin{subfigure}[b]{0.3\textwidth}
\includegraphics[width=\textwidth]{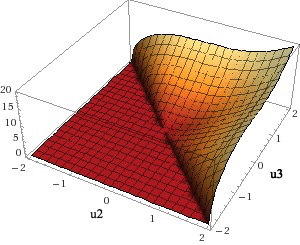} 
\end{subfigure}
\hspace{1in}
\begin{subfigure}[b]{0.32\textwidth}
\includegraphics[width=\textwidth]{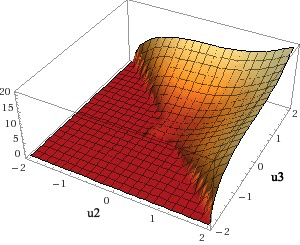}
\end{subfigure}

\vspace{.2in}

\begin{subfigure}[b]{0.32\textwidth}
\includegraphics[width=\textwidth]{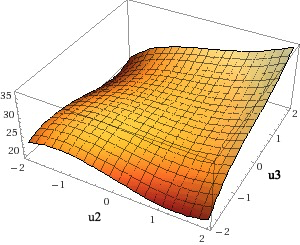}
\end{subfigure}
\hspace{1in}
\begin{subfigure}[b]{0.32\textwidth}
\includegraphics[width=\textwidth]{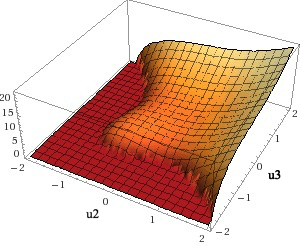}
\end{subfigure}
\caption{Real part of $\sqrt{25(u_1^3+u_2^3+u_3^3) + 21u_1u_2u_3}$ at $u_1 =0, 0.3, 1, 3$ (CW from the top left)}
 \label{fig:Xc=5}
\end{figure}

\noindent {\it Proof of Theorem~\ref{thm:geometry}}.
(a) We have that $Z$ is generated by three algebraically independent elements $u_1$, $u_2$, $u_3$ of degree 2 and one element $g$ of degree 3, subject to one relation $F$ of degree 6, by \cite[Theorems~3.7,~4.6, and~4.7]{SmithTate}. Now part (a) follows by direct computation in the algebra $S(1,1,c)$. One can do this by hand, but we execute this with the computer algebra software GAP  using the GBNP package for noncommutative Groebner basis \cite{GAP-GBNP}. 
We  check that $u_1, u_2, u_3, g$ commute with each of $x_1:=x$, $y_1:=y$, $z_1:=z$.

{\footnotesize 
\begin{multicols}{2}
\begin{verbatim}
LoadPackage( "GBNP" );
SetInfoLevel(InfoGBNP,0);
SetInfoLevel(InfoGBNPTime,0);
A:=FreeAssociativeAlgebraWithOne
     (Rationals, "x", "y", "z");
x:=A.x;; y:=A.y;; z:=A.z;; o:=One(A);; 
uerels:=[x*y+y*x+5*z*z, y*z+z*y+5*x*x, z*x+x*z+5*y*y];
uerelsNP:=GP2NPList(uerels);;
PrintNPList(uerelsNP);
GBNP.ConfigPrint(A);
GB:=SGrobner(uerelsNP);;
PrintNPList(GB);
x1:=[[[1]],[1]];;     y1:=[[[2]],[1]];; 
z1:=[[[3]],[1]];;
u1:=[[[1,1]],[1]];;   u2:=[[[2,2]],[1]];; 
u3:=[[[3,3]],[1]];;
g:=[[[2,2,2],[2,1,3],[1,2,3],[1,1,1]],[5,1,-1,-5]];;
\end{verbatim}

\begin{verbatim}

MulQA(x1,u1,GB) - MulQA(u1,x1,GB);
MulQA(x1,u2,GB) - MulQA(u2,x1,GB);
MulQA(x1,u3,GB) - MulQA(u3,x1,GB);

MulQA(y1,u1,GB) - MulQA(u1,y1,GB);
MulQA(y1,u2,GB) - MulQA(u2,y1,GB);
MulQA(y1,u3,GB) - MulQA(u3,y1,GB);

MulQA(z1,u1,GB) - MulQA(u1,z1,GB);
MulQA(z1,u2,GB) - MulQA(u2,z1,GB);
MulQA(z1,u3,GB) - MulQA(u3,z1,GB);

MulQA(x1,g,GB) - MulQA(g,x1,GB);
MulQA(y1,g,GB) - MulQA(g,y1,GB);
MulQA(z1,g,GB) - MulQA(g,z1,GB);
\end{verbatim}
\end{multicols}
}
To view {\small {\tt g}}, for instance, enter {\small {\tt PrintNP(g);}}. The output of the last twelve lines are all 0. Thus, $u_1, u_2, u_3, g$ are all central elements of $S(1,1,5)$. One can replace $c=5$ with various values of $c \neq 0, 1, -8$, and this yields the same output. 

Now to see that $F$ is the relation of $Z$, more care is needed. Enter:
\smallskip

{\footnotesize 
\begin{verbatim}
PrintNP(MulQA(g,g,GB)); 
PrintNP(MulQA(u1,MulQA(u1,u1,GB),GB));       PrintNP(MulQA(u2,MulQA(u2,u2,GB),GB));
PrintNP(MulQA(u3,MulQA(u3,u3,GB),GB));       PrintNP(MulQA(u1,MulQA(u2,u3,GB),GB));
\end{verbatim}
}
\smallskip

\noindent and compare terms to derive the coefficients of $F$ as claimed.

\medskip

(b) The arguments below are standard in ring theory and in representation theory, but we provide details for the reader's convenience. Recall from  Lemma~\ref{lem:irrep2} that all non-trivial irreducible representations of $S$ are of dimension 2.
Let maxSpec($A$) denote the set of maximal ideals of an algebra $A$. Moreover,   a {\it primitive} ideal of $A$ is an ideal that arises as the kernel of an irreducible representation of $A$; denote the set of such ideals by prim($A$). Take [Irrep($A$)] to be the set of equivalence classes of irreducible representations of~$A$. 

Since $S$ is PI, we see that there is a bijective correspondence between [Irrep($S$)] and prim($S$) as follows. Equivalent representations of $S$ have the same kernel, so we get a surjective map $\phi$: [Irrep($S$)]$\to$ prim($S$), given by $\psi \mapsto \ker \psi$.
On the other hand, take $P \in$ prim$(S)$, that is, the kernel of an irreducible representation $\psi$ of $S$. Then, $\psi$ is also an irreducible representation of $S/P$. Now $S/P \cong \text{Mat}_t(\mathbb{C})$ for $t = 1$ or 2 by \cite[Theorem~I.13.5(1)]{BGbook}, and all irreducible representations of matrix algebras are equivalent to the identity representation by the Skolem-Noether theorem. So, $P \in$ prim$(S)$ has a unique preimage $\phi^{-1}(P)$ in [Irrep($S$)].

Moreover, we see that there is a bijective correspondence  [Irrep($S$)] and maxSpec($S$) as follows. Maximal ideals are primitive. On the other hand, take $P$ a nonzero primitive ideal of $S$. Again, by \cite[Theorem~I.13.5(1)]{BGbook}, $S/P$ is isomorphic to a matrix ring, which is simple. Thus, $P$ is a maximal ideal of $S$.
So it suffices to show that the ideals of  maxSpec($S$) and of maxSpec($Z$) are in bijective correspondence.

Consider the map
$$\eta: \text{maxSpec($S$)} \to \text{maxSpec($Z$)}, \quad M \mapsto M \cap Z.$$ 
The map $\eta$ is well-defined and surjective by \cite[Proposition~III.1.1(5)]{BGbook}. Now by Lemma~\ref{lem:irrep2}, the trivial representation of $S$ corresponds to the augmentation (maximal) ideal $S_+:= (x,y,z)$ of $S$, and the set of equivalence classes of non-trivial irreducible representations of $S$ correspond to the maximal ideals $M$ of $S$ not equal to $S_+$.
Thus, $\eta(S_+) = Z_+$, and  it suffices to show that the ideals of  maxSpec($S$)$\setminus S_+$ and of maxSpec($Z$)$\setminus Z_+$ are in bijective correspondence.

Take Az($S$) to be the set of maximal ideals $\mathfrak{m}$ of $Z$ so that (i) $\mathfrak{m}= M \cap Z$ for $M \in$ maxSpec($S$), and (ii) $M$ is  the kernel of a 2-dimensional irreducible representation of $S$. Namely, Az($S$) is the {\it Azumaya locus} of $S$ over $Z$.  Consider the map 
$$\rho: \text{Az(S)} \to \text{maxSpec($S$)}, \quad \mathfrak{m} \mapsto \mathfrak{m}S.$$
We get that $\eta \rho(\mathfrak{m}) = \eta (\mathfrak{m}S) = (\mathfrak{m}S) \cap Z = \mathfrak{m}$; the last equality holds by \cite[Theorem~III.1.6(3)]{BGbook}. So, $\eta$ is bijective on $\rho(\text{Az($S$)})$.
 Since Az($S$) = maxSpec($Z$)$\setminus Z_+$ by Lemma~\ref{lem:irrep2}, and since $\rho$ is injective, we conclude that $\eta$ is bijective on maxSpec($S$)$\setminus S_+$, as desired.

\medskip

(c) To see that the claim  follows from parts (a) and (b), we have to show that the smooth locus of $X_c$ consists of all nonzero points. This is achieved by using \cite[Theorem~6.2]{Smith}; namely, we verify that  the common zero set of the vanishing of all partial derivatives of $F$ is the origin of $X_c$:
\medskip

{\footnotesize
\begin{verbatim}
   F:=g^2-c^2*(u1^3+u2^3+u3^3) - (c^3-4)*u1*u2*u3;              
   solve([diff(F,g),diff(F,u1),diff(F,u2),diff(F,u3)],[g,u1,u2,u3]);
   >                [[g = 0, u1 = 0, u2 = 0, u3 = 0]]
\end{verbatim}
}
\vspace{-.2in}

\qed

\medskip

\begin{remark}
One may push the result above further and study the {\it moduli space} (or {\it GIT quotient}) that parametrizes the set of equivalence classes of irreducible representations of $S$. But this is not the focus of this work here.
On the other hand, if one wants to understand irreducible representations of $S$ topologically, then one could consider the {\it Jacobson topology} (or {\it hull-kernel topology}) on the set prim($S$). 
\end{remark}

\begin{remark}
The following code verifies that the irreps produced in Tables~3 and~4  indeed correspond to points on $X_c$. One must first run the algorithm in the previous sections: Sections~\ref{sec:allreps} and~\ref{sec:1Jordanblock} for the one Jordan block case, and Sections~\ref{sec:allreps} and~\ref{sec:2Jordanblock} for the two Jordan block case. 

{\footnotesize 
\begin{multicols}{2}
\begin{verbatim}
# For X being 1 Jordan block
# Uncomment one of the following E
# E:=eval(NonRedFams[4],[z4=1]);
# E:=eval(NonRedFams[5],[y4=1]);
\end{verbatim}

\begin{verbatim}
## For X being 2 Jordan block
## Uncomment one of the following E
## E:=eval(NonRedFams[1],[y3=1]);
## E:=eval(NonRedFams[2],[y3=1]);
## E:=eval(NonRedFams[3],[z3=1]);
## E:=eval(NonRedFams[5],[y3=1]);
\end{verbatim}

\begin{verbatim}
U1:=Multiply(E[1][1],E[1][1]);  
U2:=Multiply(E[1][2],E[1][2]);  
U3:=Multiply(E[1][3],E[1][3]);  
C:=<<zc,0|0,zc>>; 
G:= Multiply(Multiply(Multiply(
             C,E[1][2]),E[1][2]),E[1][2])
   +Multiply(Multiply(E[1][2],E[1][1]),E[1][3])
   -Multiply(Multiply(E[1][1],E[1][2]),E[1][3])
   -Multiply(Multiply(Multiply(
             C,E[1][1]),E[1][1]),E[1][1]);
C2:=<<zc^2,0|0,zc^2>>;
C34:=<<zc^3-4,0|0,zc^3-4>>;
F:= Multiply(G,G) 
   -Multiply(C2,Multiply(Multiply(U1,U1),U1)
   +Multiply(Multiply(U2,U2),U2)
   +Multiply(Multiply(U3,U3),U3)) 
  - Multiply(C34,Multiply(Multiply(U1,U2),U3));
  
simplify(F);
\end{verbatim}

\begin{Verbatim} [baselinestretch=0.7]
 >                 [0    0]
                   [      ]
                   [0    0]
\end{Verbatim}
\end{multicols}
}
\noindent By evaluating {\tt simplify(U1);}, {\tt simplify(U2);}, {\tt simplify(U3);}, {\tt simplify(G);} for each of the six irreducible representative families above, we obtain the corresponding points on the 3-fold $X_c = \mathbb{V}(F) \subset \mathbb{C}^4_{\{u_1,u_2,u_3,g\}}$.
\end{remark}
\smallskip


\section{Irreducible representations of $\mathbb{C}_{-1}[x,y]:=\mathbb{C}\langle x,y \rangle/(xy+yx)$} \label{sec:skewpoly}

The purpose of this section is to illustrate our algorithm of Sections~\ref{sec:methodology}-\ref{sec:2Jordanblock} (Steps 0-2, 3a, 3b) by replacing the Sklyanin algebra $S(1,1,c)$ with a class of algebras that are much better understood. Here, we study irreducible representations of the skew polynomial ring: 
$$\mathbb{C}_{-1}[x,y] := \mathbb{C}\langle x,y \rangle/(xy+yx),$$
up to equivalence; these results are well-known. At the end of the section, we provide a geometric parametrization of these irreps, akin to Theorem~\ref{thm:geometry} for $S(1,1,c)$.
Now we remind the reader of a few preliminary results.

\begin{lemma} \label{lem:skew}
\begin{enumerate}
\item The 1-dimensional irreps  of $\mathbb{C}_{-1}[x,y]$ are, up to equivalence,  of the form 
\begin{equation} \label{eq:C-1x,y1}
\rho_{\alpha}: \mathbb{C}_{-1}[x,y] \to \mathbb{C}, ~~x \mapsto \alpha,  ~~y \mapsto 0, \text{ for } \alpha \in \mathbb{C}
 \quad  \text{and} \quad \rho_{\beta}: \mathbb{C}_{-1}[x,y] \to \mathbb{C}, ~~ x \mapsto 0,  ~~y \mapsto \beta,\text{ for } \beta \in \mathbb{C}.
\end{equation}
\item All irreducible representations of $\mathbb{C}_{-1}[x,y]$ are finite-dimensional, of at most dimension 2.
\end{enumerate}
\end{lemma}

\begin{proof}
(a) This follows by an easy computation.

(b) By \cite[Proposition~3.1]{BG} and \cite[Example~I.14.3(1)]{BGbook}, an irrep of $\mathbb{C}_{-1}[x,y]$ is of at most dimension 2.
\end{proof}

With the lemma above, we see that to classify irreps of $\mathbb{C}_{-1}[x,y]$, we just need to compute the 2-dimensional irreps $$\psi: \mathbb{C}_{-1}[x,y] \to Mat_2(\mathbb{C}), \quad \quad x \mapsto X, ~y \mapsto Y,$$ up to equivalence. 

Without loss of generality, we can assume that $X$ is in Jordan form, that is, either one Jordan block or  diagonal.  
For 1 Jordan block, uncomment {\tt \#}. For 2 Jordan blocks, uncomment {\tt \#\#}. Moreover, the following code was adapted from Sections~\ref{sec:methodology}-\ref{sec:2Jordanblock} by removing all lines and conditions involving the generator $z$, and by
changing the defining relations of the algebra. 
\medskip

{\footnotesize \begin{multicols}{2}
\noindent \framebox{{\bf Adapted Steps 0 and 1}}
{\footnotesize
\begin{verbatim}
restart;                 with(LinearAlgebra):
#  X:= <<x1, 0|1, x1>>:  Y:= <<y1, y3|y2, y4>>:  
## X:= <<x1, 0|0, x4>>:  Y:= <<y1, y3|y2, y4>>:  
XY:= Multiply(X,Y):      YX:= Multiply(Y,X):                                  
Eq1:= convert(XY+YX,list):    
#  Vars:= {x1,y1,y2,y3,y4}:
## Vars:= {x1,x4,y1,y2,y3,y4}:
M:= solve([op(Eq1)],Vars):
L:= []:
for i from 1 to nops([M]) do  T:=map(allvalues,{M[i]}):
for j from 1 to nops(T) do    L:=[op(L),T[j]]:
end do:       end do:
Families:=[]:
for i from 1 to nops(L) do          
Families:=[op(Families), [eval(X,L[i]),eval(Y,L[i])]]:
end do:
\end{verbatim}
}
\medskip
\end{multicols}}

\vspace{-.22in}

\begin{center}
\rule{12cm}{0.4pt}
\end{center} 

\vspace{-.12in}

{\footnotesize \begin{multicols}{2}
\noindent \framebox{{\bf Adapted Step 2}}
{\footnotesize
\begin{verbatim}
w:=<<p,q>>:        NonRedFams:=[]:
for i from 1 to nops(Families) do
Xw:=Multiply(Families[i][1],w):     
Yw:=Multiply(Families[i][2],w):     
NonRed:=solve([p*conjugate(p)+q*conjugate(q)<>0, 
     p*Xw[2][1]-q*Xw[1][1], p*Yw[2][1]-q*Yw[1][1], 
     op(Conditions)],[p,q]):
	if NonRed=[] then      
	NonRedFams:=[op(NonRedFams),[Families[i]]]:  
	end if:  end do:
\end{verbatim}
}
\medskip
\end{multicols}}

\vspace{-.22in}

\begin{center}
\rule{12cm}{0.4pt}
\end{center} 

\vspace{-.12in}

{\footnotesize \begin{multicols}{2}
\noindent \framebox{{\bf Adapted Step 3a}}
{\footnotesize
\begin{verbatim}
BetweenFams:=[]:        
ChangeVars:=[x1=u1,x4=u4,y1=v1,y2=v2,y3=v3,y4=v4]:
Q:=<<q1,q3|q2,q4>>:                 
Qi:=MatrixInverse(Q):
for i from 1 to nops(NonRedFams) do
	Xconj:=Multiply(Q,Multiply(eval(
	    NonRedFams[i][1][1],ChangeVars),Qi)):
	Yconj:=Multiply(Q,Multiply(eval(
	    NonRedFams[i][1][2],ChangeVars),Qi)):
	for j from i to nops(NonRedFams) do
		Xj:=NonRedFams[j][1][1]:   Yj:=NonRedFams[j][1][2]:             
		Equiv1:= convert(Xj-Xconj,list):    
		Equiv2:= convert(Yj-Yconj,list):     
		Conditions:= [q1*q4-q2*q3<>0]:
		Equiv:= solve([op(Equiv1),op(Equiv2),op(Conditions)]):
		BetweenFams:=[op(BetweenFams),[i,j,Equiv]]:
	end do:     end do:
\end{verbatim}
}
\medskip
\end{multicols}}

\vspace{-.22in}

\begin{center}
\rule{12cm}{0.4pt}
\end{center} 

\vspace{-.12in}

{\footnotesize \begin{multicols}{2}
\noindent \framebox{{\bf Adapted Step 3b}}
{\footnotesize
\begin{verbatim}
IrConditions:=[]:
for i from 1 to nops(NonRedFams) do
Xw:=Multiply(NonRedFams[i][1][1],w): 
Yw:=Multiply(NonRedFams[i][1][2],w): 
Ir:=solve([p*conjugate(p)+q*conjugate(q)<>0,
     p*Xw[2][1]-q*Xw[1][1],p*Yw[2][1]-q*Yw[1][1]]):
IrConditions:=[op(IrConditions),[i,Ir]]:
end do:
\end{verbatim}
}
\medskip
\end{multicols}}
\medskip

We obtain the result below.

\begin{proposition} \label{prop:skew}
All irreducible representations $\phi$ of $\mathbb{C}_{-1}[x,y]$ are of dimensions 1 or 2. In dimension 1, irreps are of the form \eqref{eq:C-1x,y1}. In dimension 2, all irreps, up to equivalence, take the form 
\begin{equation} \label{eq:C-1x,y2}
\psi_{\alpha, \beta} : \mathbb{C}_{-1}[x,y] \longrightarrow Mat_2(\mathbb{C}), \quad x\mapsto \begin{pmatrix} -\alpha & 0\\ 0 & \alpha \end{pmatrix}, 
~~y\mapsto \begin{pmatrix} 0& 1\\ \beta & 0 \end{pmatrix}, 
\end{equation}
\text{ for $\alpha, \beta \in \mathbb{C}$ with $\alpha \beta \neq 0$.}
\end{proposition}

\noindent {\it Proof}.
The first two statements follow from Lemma~\ref{lem:skew}. To get the last statement, we run the adapted algorithm above. We only obtain reducible representations in the one Jordan block case; just enter {\tt NonRedFams;} and {\tt IrConditions;} to see this. 

On the other hand in the two Jordan block case, we first print off {\tt NonRedFams}.

\medskip

\footnotesize{
\begin{Verbatim}[baselinestretch=0.7]
    [0    0]  [y1    y2]          [x1    0]  [0    0 ]          [0    0 ]  [y1    0]     
  [[[      ], [        ]]]      [[[       ], [       ]]]      [[[       ], [       ]]] 
    [0    0]  [y3    y4]          [0     0]  [0    y4]          [0    x4]  [0     0]     
    
    [x1    0 ]  [0    0]          [-x4    0 ]  [0     y2]
  [[[        ], [      ]]]      [[[         ], [        ]]]
    [0     x4]  [0    0]          [ 0     x4]  [y3    0 ]
\end{Verbatim}
}

\medskip

\noindent {\normalsize Consider the following snippets of output from  {\tt BetweenFams}. \par}

\medskip

{\footnotesize
\begin{Verbatim}[baselinestretch=0.7]
   [2, 3, {q1 = 0, q2 = q2, q3 = q3, q4 = 0, u1 = x4, v4 = y1, x4 = x4, y1 = y1}]

                                                            y3 q2       y2 q3                                    
   [5, 5, {q1 = 0, q2 = q2, q3 = q3, q4 = 0, u4 = -x4, v2 = -----, v3 = -----, x4 = x4, y2 = y2, y3 = y3}]
                                                             q3          q2                                    
 \end{Verbatim}
}

\medskip

\noindent {\normalsize Therefore, any member of {\tt NonRedFams[3]} is equivalent to a member of {\tt NonRedFams[2]}. 
\medskip

$\bullet$ So, {\tt NonRedFams[3]} is removed from our consideration.
\medskip

\noindent Moreover, {\tt NonRedFams[5]} forms an equivalence family as $x_4, y_2, y_3$ are free.

Take into consideration the output from {\tt IrConditions}.\par}

\medskip

 {\footnotesize
\begin{Verbatim}[baselinestretch=0.7]
                               2                2
                           y3 p  + p y4 q - y2 q
   [1, {p = p, q = q, y1 = ----------------------, y2 = y2, y3 = y3, y4 = y4}]
                                    p q

   [2, {p = p, q = 0, x1 = x1, y4 = y4}]       
   
   [4, {p = p, q = 0, x1 = x1, x4 = x4}]

   [5, {p = p, q = 0, x4 = x4, y2 = y2, y3 = 0},   {p = 0, q = q, x4 = x4, y2 = 0, y3 = y3},

                                       2
                                   y3 p
       {p = p, q = q, x4 = 0, y2 = -----, y3 = y3}]
                                     2
                                    q
\end{Verbatim}
}

\medskip

{\normalsize Now, we can conclude that {\tt NonRedFams[1]}, {\tt NonRedFams[2]}, {\tt NonRedFams[4]} consist of reducible representations, so these families are eliminated from our consideration. Further, {\tt NonRedFams[5]} forms an irreducible representative family with $y_2 =1$; we can see this by adapting and running the algorithm for Step 3b in Section~\ref{sec:2Jordanblock} in this case. \qed \par}
\medskip

{\normalsize The geometric parameterization  of the equivalence classes of irreducible representations of $\mathbb{C}_{-1}[x,y]$ is given as follows; see also Figure~2.

\begin{corollary}  \label{cor:skew} We have the following statements.
\begin{enumerate}
\item We have that the center $Z$ of $\mathbb{C}_{-1}[x,y]$ is the commutative polynomial ring generated by $u_1:=x^2$ and $u_2:=y^2$.
\item The set of equivalence classes of irreducible representations of $S$ are in bijective correspondence with the set of maximal ideals of $\mathbb{C}[x^2, y^2]$. 
\item The geometric parameterization of the set of equivalence classes of irreducible representations of $\mathbb{C}_{-1}[x,y]$ is the 2-dimensional affine space $\mathbb{C}^2_{\{u_1, u_2\}}$. In particular: 
\begin{itemize} 
\item points of $\mathbb{C}^2 \setminus \mathbb{V}(u_1u_2)$ correspond to irreducible 2-dimensional representations of $\mathbb{C}_{-1}[x,y]$, 
\item points on the axes $\mathbb{V}(u_1u_2)$ not equal to the origin correspond to non-trivial 1-dimensional representations of $\mathbb{C}_{-1}[x,y]$, and 
\item the origin corresponds to the trivial representation of $\mathbb{C}_{-1}[x,y]$.
\end{itemize}
\end{enumerate}
\end{corollary}

\begin{proof}
(a) The algebra $\mathbb{C}_{-1}[x,y]$ has a $\C$-vector space basis given by $\{x^i y^j | i, j \in \mathbb{N}\}$.  Since $(x^i y^j) x = (-1)^j x^{i+1} y^j = x^{i+1} y^j$ and $y(x^i y^j) = (-1)^i x^i y^{j+1} = x^i y^{j+1}$ implies that $i, j$ are even, the result is clear.
\medskip

(b) This follows by the proof of Theorem~\ref{thm:geometry}(b).
\medskip

(c) The first statement follows as Spec($Z$) = $\mathbb{C}^2_{\{u_1, u_2\}}$. Now the remaining statements hold by~\eqref{eq:C-1x,y2} and~\eqref{eq:C-1x,y1} where $u_1 = \alpha^2$ and $u_2 = \beta$.
\end{proof}

\begin{figure}[h] 
\centering

\includegraphics[width=0.32\textwidth]{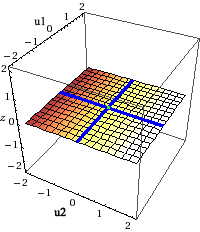} 
\caption{Affine 2-space parametrizing irreps of $\C_{-1}[x,y]$; axes parametrize 1-dimensional irreps}
\end{figure} 

\section*{Acknowledgments}
This project was inspired by the second author's conversation with Andrew Morrison and Bal\'{a}zs Szendr\H{o}i at the Interactions between Algebraic Geometry and Noncommutative Algebra workshop at the  Mathematisches Forchungsinstitut Oberwolfach in May 2014.  The authors are grateful to the referees of our of submission, particularly to a referee who discussed with us a Clifford-theoretic approach to this problem (as mentioned in   Remark \ref{rem:intro}).
Travel for the first author was supported by an Undergraduate Research Supplement of  the second author's National Science Foundation grant DMS-1550306.


\bibliography{Irreps-bibtex}

\end{document}